\documentclass[compress,final,3p,times,11pt]{elsarticle}
\usepackage{amsfonts}
\usepackage{dsfont}
\usepackage{amssymb}
\usepackage{amsmath}
\usepackage{color,epstopdf}
\usepackage{mathrsfs}
\usepackage{graphicx,color,epstopdf}
\usepackage{float}
\usepackage{caption}
\usepackage{subcaption}
\usepackage{float}
\usepackage{subcaption}
\newtheorem{proof}{Proof}
\newtheorem{fact}{Fact}
\newtheorem{assumption}{Assumption}

\newtheorem{theorem}{Theorem}
\newtheorem{remark}{Remark}

\journal{.}

\begin{document}
\title{Dynamic Systems Framework for Modeling COVID-19 with Lévy Noise}

\author[addr1]{Daniel Tesfay}\corref{cor1}\ead{daniel.tesfay@math.uu.se}\cortext[cor1]{Corresponding author}
\author[addr2]{Almaz Tesfay}\ead{almaz.abebe@howard.edu}
\author[addr3]{James Brannan}\ead{jrbrn@clemson.edu}
\address[addr1]{Department of Mathematics, Uppsala University, Box 480, 751 06, Uppsala, Sweden}
\address[addr2]{Department of Mathematics, Howard University, Washington DC, 20059, USA}
\address[addr3]{Department of Mathematical Sciences, Clemson University, Clemson, South Carolina 29634, USA}

\begin{abstract}
Natural fluctuations have played a crucial role in affecting the dynamics of pervasive diseases such as the coronavirus. Examining the effects of irregular unsettling disturbances on epidemic models is important for understanding these dynamics. In this study, we introduce a mathematical model for the SIR (Susceptible-Infectious-Recovered) dynamics of the coronavirus, incorporating perturbations in the contact rate through L\'evy noise.  The utilization of the L\'evy process is essential for the protection and control of diseases. We delve into the dynamics of both the deterministic model and the global positive solution of the stochastic model, establishing their existence and uniqueness. Additionally, we explore conditions for the termination and persistence of the infection. Furthermore, we derive the basic reproduction number, a critical determinant of disease extinction or persistence. Numerical results indicate that COVID-19 dissipates from the population when the reproduction number is less than one in the presence of significant or minor noise. 
Conversely, controlling epidemic diseases becomes challenging when the reproduction number exceeds one.
To illustrate this phenomenon, we provide numerical simulations, offering insights into the dynamics of the disease and the efficacy of control measures.
\end{abstract}
\begin{keyword}

stochastic COVID-19, Brownian motion, L\'evy noise,  extinction, persistence. 
\MSC[2020]-Mathematics Subject Classification: 39A50, 45K05, 65N22.
\end{keyword}
\maketitle

\section{Introduction}
Coronaviruses constitute an extensive family of viruses typically responsible for causing mild to moderate upper-respiratory tract illnesses. Various coronaviruses circulate among animals, including pigs, cats, and bats. On occasion, these viruses can jump from animals to humans, leading to infections. Some of these infections have resulted in severe outbreaks, such as the SARS coronavirus (SARS-CoV). In December 2019, a novel coronavirus, COVID-19, emerged at a seafood market in Wuhan, China. By March 2020, the World Health Organization (WHO) declared the virus to be a global pandemic. COVID-19 primarily spreads through respiratory droplets or nasal discharge when an infected individual coughs or sneezes. Mathematical models play a crucial role in understanding and describing infectious diseases, both in theory and practical applications (see, for example, [6–8, 18, 34]). Developing and analyzing such models contribute significantly to unraveling transmission mechanisms and disease characteristics. Consequently, these insights enable the formulation of effective strategies for prediction, prevention, and control of infections, ensuring the well-being of populations.

To date, a multitude of mathematical models describing infectious diseases through differential equations have been formulated and scrutinized to understand the dynamics of infection spread, exemplified by research on \cite{bocharov2018mathematical, Brauer, Brauer1, Martcheva}. Recently, the mathematical modeling of the COVID-19 pandemic has captivated the attention of numerous experts, including mathematicians, scientists, epidemiologists, pharmacists, and chemists. The outcomes of these endeavors have yielded several noteworthy and crucial results, as highlighted in the works of \cite{Martcheva, Higazy, ccakan2020dynamic, Kumar, Li}
and references therein. The exploration of numerical models for the COVID-19 pandemic is proving instrumental in advancing our comprehension of the disease's trajectory and guiding informed decision-making in the fields of public health and medicine.

Environmental fluctuations have emerged as a significant factor in the study of diseases, particularly in the context of the coronavirus. Consequently, it becomes crucial to investigate the impact of random disturbances on epidemic models. Disease spread is inherently stochastic, and the introduction of stochastic noise can notably influence the likelihood of disease extinction during the early stages of an outbreak. While ordinary differential equation (ODE) models provide specific sample solutions, employing a stochastic differential equation (SDE) model allows the exploration of the stochastic distribution of disease dynamics.

In the current landscape, various stochastic epidemic models have been extensively explored. Notably, articles such as those in the \cite{2020A,tesfay2020dynamical,Amu1,Amu2,Amu3,zhou2016threshold,Amu4, Amu5,Amu6} series have delved into stochastic epidemic models influenced by Lévy noise. Motivated by this body of research, we introduce the assumption that the contact rate is perturbed by Lévy noise, emphasizing the need to employ Lévy processes for disease protection and control. The resulting model is detailed in  (\ref{CV-stoch}), as elaborated in Section \ref{stoch}. Our objective is to derive the basic reproduction number, a critical determinant governing the extinction or persistence of the disease (infection). The reproduction number represents the average number of secondary infections produced by a single infected individual in a susceptible population, making it an essential metric for understanding disease spread dynamics. If the reproduction number exceeds 1, it indicates that each existing infection leads to more than one new infection, resulting in exponential growth of the disease within the population. Conversely, if the reproduction number is less than 1, the disease is likely to diminish and eventually disappear, as each infected person infects, on average, fewer than one new individual. This research contributes to the broader understanding of stochastic epidemic modeling and its implications for disease dynamics.

The subsequent sections of the paper unfold as follows: Section \ref{determ}: Model formulation- In this section, we meticulously detail the formulation of the model, providing a comprehensive overview of its deterministic aspects. Section \ref{dynamics}: Dynamics of the deterministic model- we discuss the reproduction number and stability of the system. Section \ref{stoch}: Formulation and description of stochastic COVID-19 model- we explore and elucidate the dynamic properties and behaviors inherent in the stochastic aspects of the model. Section \ref{Numb}: Numerical experiments- this section is dedicated to presenting the numerical solutions derived for the proposed model. Through numerical simulations, we offer insights into the practical implications and outcomes of the model. Section \ref{Conc}: Conclusion - we present the conclusive remarks and findings of the entire paper. This section serves to summarize key results, implications, and potential avenues for future research.


\section{Model formulation}\label{determ}
The deterministic COVID-19 model is
\begin{align}\label{CV-deter}
 &\frac{dS_t}{dt}= \theta-\xi\,S_t\,I_t-\eta\,S_t+\rho \,R_t\nonumber\\
 &\frac{dI_t}{dt}= \xi\,S_t\,I_t-(\eta+\gamma)\,I_t\nonumber\\
 &\frac{dR_t}{dt}=\gamma\,I_t-(\eta+\rho)\,R_t
\end{align}
Where, the variable $\theta$ represents the enrollment rate into the population, reflecting the influx of individuals who directly become part of the susceptible class $S_t$. This rate is fundamental in understanding the pool of individuals susceptible to the disease at any given time.
The variable $\xi$ stands for the contact rate, which primarily signifies the incidence rate at which individuals from the susceptible class transition to the infectious class $I_t$. This rate captures the frequency and intensity of interactions that facilitate disease transmission within the population.
Meanwhile, the variable $\eta$ encompasses the outgoing rate from each disease class, incorporating various factors such as natural death, deaths due to the disease (in this case, corona), or migration rates for individuals within each respective class. It provides insights into the dynamics of population turnover and mortality related to the disease. 
Additionally, $\gamma$ signifies the recovery rate, indicating the speed at which individuals move from the infectious class $I_t$ to the recovered class $R_t$. This parameter reflects the effectiveness of treatment, immunity development, or other factors influencing recovery from the disease. 
The term $\rho$  captures the relapse rate, illustrating the rate at which individuals who have recovered from the disease and entered the recovered class $R_t$ revert back to the susceptible class $S_t$. This aspect is crucial in understanding the potential recurrence or persistence of the disease within the population over time.
\subsection{Non-negativity of the solutions}
This subsection is dedicated to exploring positive solutions for future time, considering their respective initial conditions. The focus here is on understanding and analyzing the behavior of solutions as they evolve over time, shedding light on the dynamics and outcomes based on the given initial conditions.
\begin{theorem}[Positivity]
Consider any initial data $(S_0, I_0, R_0)$ belonging to the set $\mathbb{R}_{+}^3$, and let $\big(S_t, I_t, R_t\big)$ represent the solution corresponding to this initial data. The assertion is that the set $\mathbb{R}_{+}^3$ is a positively invariant set of the model \eqref{CV-deter}.
\end{theorem}
In other words, all solution trajectories of \eqref{CV-deter} with initial conditions $S(0) = S_0, \; I(0) = I_0,$ and $R(0) = R_0$ in the first octant of the SIR-space remain in the first octant. 

This property underscores the positivity and non-negativity of the solutions, ensuring that the variables $S_t, I_t$, and $R_t$ remain non-negative throughout the dynamics of the system.
\begin{proof}

By letting $\alpha=\xi\,I_t$, 
and 
, first equation of \eqref{CV-deter}
implies that
\begin{align*}
 \frac{dS_t}{dt} &= \theta-\alpha\,S_t-\eta\,S_t+\rho \,R_t, \\
 \mbox{or} \hspace{2cm} \frac{dS_t}{dt} &\geq \theta-\alpha\,S_t-\eta\,S_t,
\end{align*}
For every time $t\in [0,+\infty[$, where solution of the system \eqref{CV-deter} exists, the above differential inequality has the solution as
\begin{eqnarray*}
\frac{dS_t}{dt}+\alpha\,S_t+\eta\,S_t&\geq& \theta,\\
\frac{dS_t}{dt}+(\alpha+\eta)\,S_t&\geq& \theta,\\
\frac{d}{dt}\bigg(S_te^{\eta t+\int_0^t\alpha(s)ds}\bigg)&\geq&
\theta e^{\eta t+\int_0^t\alpha(s)ds},
\end{eqnarray*}
then it follows that
\begin{eqnarray*}
S_te^{\eta t+\int_0^t\alpha(s)ds}-S_0&\geq&\int_0^t\theta
e^{\eta w+\int_0^w\alpha(u)du}dw,\\
S_te^{\eta t+\int_0^t\alpha(s)ds}&\geq&S_0+\int_0^t\theta
e^{\eta w+\int_0^w\alpha(u)du}dw,\\
S_t&\geq&S_0e^{-\eta t-\int_0^t\alpha(s)ds}+e^{-\eta
t-\int_0^t\alpha(s)ds}\times\int_0^t\theta e^{\eta
w+\int_0^w\alpha(u)du}dw>0.
\end{eqnarray*}
Hence, $S_t$ is positive for all values of $t$ in the considered interval. As second equation of the system \eqref{CV-deter} implies
that
\begin{align*}
\frac{dI_t}{dt}&\geq-(\eta+\gamma)\,I_t\\
\frac{dI_t}{\,I_t}&\geq -(\eta+\gamma)dt,\qquad (I_t\neq0).
\end{align*}
By integration, we have
\begin{equation*}
\ln{I_t}\geq -(\eta+\gamma)t+C,
\end{equation*}
\begin{equation*}
I_t\geq e^{-(\eta+\gamma)t+C},
\end{equation*}
\begin{equation*}
I_t\geq C_1e^{-(\eta+\gamma)t},
\end{equation*}
at $t=0$,
\begin{equation*}
I_t\geq I_0e^{-(\eta+\gamma)0}\geq 0.
\end{equation*}
Hence, $I_t$ is positive for all values of $t$. Also, from last
equation of the system \eqref{CV-deter} we have
\begin{equation*}
\frac{dR_t}{dt}=\gamma\,I_t-(\eta+\rho)\,R_t,
\end{equation*}
or
\begin{equation*}
\frac{dR_t}{dt}\geq -(\eta+\rho)\,R_t,
\end{equation*}
can be written as
\begin{equation*}
\frac{dR_t}{R_t}\geq -(\eta+\rho) dt,~~~(R_t\neq0).
\end{equation*}
By integrating, we get
\begin{equation*}
\ln{R_t}\geq -(\eta+\rho) t+K,
\end{equation*}
\begin{equation*}
R_t\geq e^{-(\eta+\rho) t+K},
\end{equation*}
\begin{equation*}
R_t\geq K_1e^{-(\eta+\rho) t},
\end{equation*}
at $t=0$,
\begin{equation*}
R_t\geq R_0e^{-(\eta+\rho) 0}\geq 0.
\end{equation*}
Hence, $R_t$ also positive in the given interval.

The proof is complete.  $\Box$
\end{proof}

\section{Dynamics of the deterministic model}\label{dynamics}
\subsection{
Virus-free  equilibrium point and reproduction number}
Consider the following algebraic system for finding the equilibria of model (\ref{CV-deter}),
\begin{eqnarray}\label{eq:6}
\theta-\xi\,S_t\,I_t-\eta\,S_t+\rho \,R_t&= 0,\nonumber\\
\xi\,S_t\,I_t-(\eta+\gamma)\,I_t &= 0,\\
\gamma\,I_t-(\eta+\rho)\,R_t &= 0.\nonumber
\end{eqnarray}
By some algebraic manipulations, we obtain two solutions of the system (\ref{eq:6}), one of them is the virus-free equilibrium point 

%

\begin{equation*}
\bigg(S^0, I^0, R^0\bigg)=\bigg(\frac{\theta}{\eta}, 0, 0\bigg).
\end{equation*}
Further, computing the reproduction number of  model (\ref{CV-deter}). Let $y=(I, S)$ and rewrite the model (\ref{CV-deter}) for susceptible and infected classes as in the general form
\begin{equation}\label{eq:9}
\frac{dy}{dt} = \mathcal{F}(y) - \mathcal{V}(y),
\end{equation}
where
\begin{equation}\label{eq:10}
\mathcal{F}(y) =\left(\begin{array}{cc}
\xi\,S_t\,I_t \\
0 \\
\end{array}
\right)
\qquad \qquad \mathcal{V}(y) =\left(
\begin{array}{cc}
(\eta+\gamma)\,I_t \\
-\theta+\xi\,S_t\,I_t+\eta\,S_t-\rho \,R_t\\
\end{array}
\right).
\end{equation}
Now the Jacobian of the above matrices at $E_0 = \left(
\begin{array}{cccc}
I_0\\
S_0\\
\end{array}
\right)$ are
\begin{equation*}
\mathcal{\texttt{F}}(E_0) =\left(
\begin{array}{cccc}
\frac{\xi\theta}{\eta}& 0\\
0 &0\\
\end{array}
\right)
\quad \qquad \mathcal{\texttt{V}}(E_0)=\left(
\begin{array}{cccc}
\eta+\gamma&o\\
\frac{-\xi\theta}{\eta^2(\gamma+\eta)}&\frac{1}{\eta}\
\end{array}
\right).
\end{equation*}
Hence, the reproduction number of the model (\ref{CV-deter}) can be determined by
\begin{equation}\label{eq:11}
\psi_0 = \rho(\texttt{F}{\texttt{V}}^{-1}) =
\frac{\xi\theta}{\eta(\gamma+\eta)}.
\end{equation}
\subsection{Positive endemic equilibrium point}
\begin{theorem}
For system (\ref{CV-deter}), there exists a
 unique positive endemic equilibrium point $E_+$, if $\psi_0>1.$
\end{theorem}
\begin{proof}
By some algebraic manipulations, the second solution of the system
(\ref{eq:6}) is
\begin{eqnarray*}
S_+&=&\frac{\eta+\gamma}{\xi}\\
I_+&=&\frac{(\eta+\rho)(\eta+\gamma)}{\xi(\eta+\rho+\gamma)}(\psi_0-1)\\
R_+&=&\frac{\gamma(\eta+\gamma)}{\xi(\eta+\rho+\gamma)}(\psi_0-1).
\end{eqnarray*}
So
\begin{equation*}
\bigg(S_+,I_+,R_+\bigg)=\bigg(\frac{\eta+\gamma}{\xi},\frac{(\eta+\rho)(\eta+\gamma)}{\xi(\eta+\rho+\gamma)}(\psi_0-1),\frac{\gamma(\eta+\gamma)}{\xi(\eta+\rho+\gamma)}(\psi_0-1)\bigg)
\end{equation*}
It is clear from the values of $S_+$, $I_+$ and $R_+$ that there exists a unique positive endemic equilibrium point $E_+$, if
$\psi_0>1.$  $\Box$
\end{proof}
\subsection{Stability analysis of the model}
\begin{theorem}\label{maintheorem1}
The system (\ref{CV-deter}) is locally stable related to virus-free equilibrium point $E_0$, if $\psi_0<1$ and unstable if $\psi_0>1$.
\end{theorem}
\begin{proof}
For local stability the Jacobian of system (\ref{CV-deter}) is
\begin{equation}\label{eq:j1}
  J=\left(
  \begin{array}{cccc}
    -\eta- \xi I_t& -\xi S_t &\rho\\
    \xi I_t& \xi S_t-(\eta+\gamma)& 0 \\
    0 & \gamma & -\rho-\eta \\
  \end{array}
\right).
\end{equation}
At $E_{0}$, the Jacobian reduces to
\begin{equation}\label{eiegenvalues}
  J(E_{0})=\left(
  \begin{array}{cccc}
    -\eta& -\frac{\xi\theta}{\eta} &\rho\\
    0& \frac{\xi\theta}{\eta}-(\eta+\gamma)& 0 \\
    0 & \gamma & -\rho-\eta \\
  \end{array}
\right).
\end{equation}
The eigenvalues of (\ref{eiegenvalues}) are $$\lambda_{1}= -\eta<0,$$
$$\lambda_{3}= -\rho-\eta<0,$$ and
$$\lambda_{2}=\frac{\xi\theta}{\eta}-(\eta+\gamma)=(\eta+\gamma)\bigg(\frac{\xi\theta}{\eta(\eta+\gamma)}-1\bigg) = (\eta+\gamma)(\psi_0 - 1).$$ Therefore $\lambda_{2}<0$, if $\psi_0<1$. So the system (\ref{CV-deter}) is locally stable for $\psi_0<1$ and unstable for $\psi_0>1$. The proof is complete. $\Box$
\end{proof}
In stability analysis, DFE stands for "Disease-Free Equilibrium" point. It refers to a critical point in mathematical models used to describe the dynamics of infectious diseases.
\begin{theorem}\label{maintheorem01}
If $\psi_0>1$, then the DFE point 
of the system (\ref{CV-deter}) is globally asymptotically 
stable.\end{theorem}
\begin{proof}
For the proof of this theorem, first we construct the Lyapunov function $L(S_t,I_t)$ as,
\begin{equation}\label{eqg1}
 L(S_t,I_t)=\bigg(S_t-\beta-\beta\,\ln\frac{ S_t}{\beta}\bigg)+(I_t-1-\ln I_t).
\end{equation}
Differentiating the equation (\ref{eqg1}) with respect to time, we have
\begin{equation}
   \frac{d}{dt}L(S_t,I_t)=\bigg(1-\frac{\beta}{S_t}\bigg)\frac{dS_t}{dt}+\bigg(1-\frac{1}{I_t}\bigg)\frac{dI_t}{dt},
\end{equation}
by substitution the values of $S_t$, $I_t$ and manipulating along the point $E_0$, we have
\begin{eqnarray*}
\frac{d}{dt}(L(S_t,I_t))&=&(1-\frac{\beta}{S_t})(\theta-\eta\,S_t)-\xi\,S_t+(\eta+\gamma)\\
&=&(\theta-\eta\,S_t-\frac{\beta\theta}{S_t}+\beta\eta)-\xi S_t+(\eta+\gamma)\\
&=&\theta-\theta-\beta\eta+\beta\eta-\frac{\xi\theta}{\eta}+(\eta+\gamma)\\
 &=&-\frac{\xi\theta}{\eta}+(\eta+\gamma),\\
&=&-(\eta+\gamma)\bigg(\frac{\xi\theta}{\eta(\eta+\gamma)}-1\bigg),\\
&=&-(\eta+\gamma)\bigg(\psi_0-1\bigg),\\
&\leq&0~~~~~~~for~~~\psi_0>1.
\end{eqnarray*}
Therefore, if $\psi_0>1$, then $\frac{d}{dt}(L(S_t,I_t))<0$, which implies that the DFE point of the system (\ref{determ}) is globally asymptotically 
stable for $\psi_0>1$. $\Box$
\end{proof}
\begin{remark}
An intriguing mathematical problem lies in the analysis of the stability of $E_+$. However, our primary emphasis is directed toward the case where $\psi_0 < 1$, with the overarching goal of devising an effective strategy for disease prevention. This specific scenario, where the initial reproduction number $\psi_0$ is less than one, holds significance as it indicates a potential for successful disease control and containment. Our exploration in this context is geared towards understanding the stability dynamics within the given parameters, offering valuable insights for the formulation of strategic interventions to curb the spread of the disease.
\end{remark}
\section{Formulation and description of stochastic COVID-19 model}\label{stoch}
Inspired by the notable works referenced as \cite{tesfay2020dynamical,zhang2020dynamics}, this paper embarks on the investigation of the stochastic COVID-19 epidemic model. Our premise assumes that environmental fluctuations affecting the contact rate parameter $\xi$ follow a dynamic evolution, where $\xi$ transitions to $\xi \longrightarrow \xi +\int_{\mathbb{Z}} \epsilon(u)\,\bar{\pi}(dt,du)$. Here, $\epsilon(u)$ is a bounded function constrained within the interval $0<\epsilon(u)+1$ particularly on intervals $1 \leq |u|$ or $1> |u|$. The term $\bar{\pi}(t,du)$ represents the compensated Poisson random measure, defined as $\bar{\pi}(t,du)={\pi}(t,du)-\nu(du)dt$, where $\nu(.)$ is a $\delta$-finite measure on a measurable subset $\mathbb{Z}$ of $(0,\infty)$ and $\infty>\nu(\mathbb{Z})$. Concurrently, ${\pi}(t,du)$ denotes the independent Poisson random measure on $\mathbb{Z}^+\times{{\mathbb{R}}\setminus{\{0\}}}$.
This modeling approach allows for the incorporation of environmental variability into the contact rate, offering a more nuanced understanding of the stochastic dynamics underlying the COVID-19 epidemic.

In this study, we explore a nonlinear stochastic COVID-19 system, incorporating the influence of non-Gaussian noise. The presence of non-Gaussian noise adds a layer of complexity to the modeling framework, allowing for a more realistic representation of the uncertainties and random fluctuations inherent in the dynamics of the COVID-19 epidemic. This consideration is crucial for a comprehensive understanding of the system's behavior and its response to unpredictable environmental factors.
\begin{align}\label{CV-stoch}
 &dS_t= (\theta-\xi\,S_t\,I_t-\eta\,S_t+\rho \,R_t)dt-\int_{\mathbb{Z}} \epsilon(u)S_{t^{-}}\,I_{t^{-}}\,\bar{\pi}(dt,du)\nonumber\\
 &dI_t=( \xi\,S_t\,I_t-(\eta+\gamma)\,I_t)dt+\int_{\mathbb{Z}} \epsilon(u)S_{t^{-}}\,I_{t^{-}}\,\bar{\pi}(dt,du)\nonumber\\
 &dR_t=(\gamma\,I_t-(\eta+\rho)\,R_t)dt,
\end{align}
where $S_{t^{-}}$ and $I_{t^{-}}$ are the left limits of $S_t$ and $I_t$, respectively. 

Notably, the first two equations of the stochastic COVID-19 model (\ref{CV-stoch}) exhibit independence from the third equation. Therefore, our focus shifts to the following equation:
\begin{align}\label{stoch1}
 &dS_t= (\theta-\xi\,S_t\,I_t-\eta\,S_t+\rho \,R_t)dt-\int_{\mathbb{Z}} \epsilon(u)S_{t^{-}}\,I_{t^{-}}\,\bar{\pi}(dt,du)\nonumber\\
 &dI_t=( \xi\,S_t\,I_t-(\eta+\gamma)\,I_t)dt+\int_{\mathbb{Z}} \epsilon(u)S_{t^{-}}\,I_{t^{-}}\,\bar{\pi}(dt,du)
\end{align}
\subsection{Existence and uniqueness of the solution}\label{s3}
Throughout this paper, the following notations and assumptions are employed for clarity and consistency in the exposition.
\begin{description}
   \item $\mathbb{R}^2_{+}:=\{x=(x_1,x_2)\in \mathbb{R}^2: x_i\geq 0, \, i=1,2\}$, \,\,$\mathbb{R}^3_{+}:=\{x=(x_1,x_2,x_3)\in \mathbb{R}^3: x_i\geq 0, \, i=1,2,3\}$, \,\,\, $\mathbb{R}_+=(0,\infty)$, for all $x>0$, $x-1-\ln x \geq0$ is true.
\end{description}  
\begin{assumption}\label{as1}
 $\epsilon(u)$ is a bounded function with $\epsilon(u)> -1$ \, $u\in \mathbb{Z}$ of $(0,\infty)$, there exists $K > 0$ such that $\int_{\mathbb{Z}} (\ln(1+\epsilon(u)))^2\,\nu(du) < K$;
\end{assumption}
\begin{remark}
The aforementioned assumption suggests that the intensities of jump noises are finite, ensuring a well-behaved and manageable behavior in the modeling framework.
\end{remark}

Next, we aim to demonstrate that the stochastic COVID-19 model (\ref{stoch1}) possesses a unique, positive, and globally defined solution for the initial conditions $(S(0),I(0))\in \mathbb{R}_+^2.$ This analysis underscores the existence and stability of the solution across the entire domain, providing a foundation for understanding the dynamics of the system.
\begin{theorem}\label{Ext.Soln}
If Assumption (\ref{as1}) holds, then for any given initial value $(S(0),I(0))\in \mathbb{R}^2_+$ and time $t\geq 0$, there is a unique global positive solution $(S(t),I(t))\in \mathbb{R}_+^2$ of the model (\ref{stoch1}).
\end{theorem}

\begin{proof}
The stochastic COVID-19 model (\ref{stoch1}) is characterized by locally Lipschitz continuous coefficients, ensuring the existence of a unique local solution $(S(t),I(t))\in \mathbb{R}_+^2$
defined for $t\in [0,\tau_{e})$, where $\tau_{e}$ is the time for the noise explosion. To establish the global solution, it is imperative to demonstrate that $\tau_{e} = \infty$ almost surely.

Let $m_0$ be a very large positive number 
and consider the initial conditions $(S(0), I(0)) \in \Big[ \frac{1}{m_0}, m_0\Big]$. For every integer $m > m_0$, we define the stopping time as:


\begin{align*}
   \tau_e=inf\{t\in[0,\tau_{e}):min(S(t),I(t))\leq \frac{1}{m_0},\,\, or\,\,\, max(S(t),I(t))\geq m\}
\end{align*}
As $m$ goes to $\infty$, $\tau_m$ increases. Define $lim_{m\rightarrow \infty}\,\tau_m=\tau_{\infty}$ with $\tau_{\infty} \leq \tau_e$. If we can prove that $\tau_{\infty}=\infty$ almost surely, then $\tau_e=\infty.$ If this is false, then there are two positive constants $T>0$ and $\delta\in (0,1)$ such that
\begin{align*}
\mathbb{P}\{\tau_\infty\leq T\} > \delta.
\end{align*}
Thus there is $m_1\geq m_0$ that satisfies
\begin{align*}
\mathbb{P}\{\tau_k\leq T\}\geq \delta, \quad m\geq m_1.
\end{align*}
Construct a $C^2$-function $V$: $\mathbb{R}_+^2 \rightarrow \mathbb{R}_+$ by
\begin{align}\label{C^2-Fun}
V(S,I)=(S-\beta-\beta\,\ln\frac{ S}{\beta})+(I-1-\ln I)
\end{align}
Applying  It$\hat{o}$ formula to Eq. (\ref{C^2-Fun}), obtain
\begin{align}
   dV(S,I) &= d\left[\left(S-\beta-\beta\,\ln\frac{ S}{\beta}\right)+(I-1-\ln I)\right]\nonumber\\
   &=\left(1-\frac{\beta}{S}\right)[(\theta-\xi\,S_t\,I_t-\eta\,S_t+\rho \,R_t)dt]-\left(1-\frac{\beta}{S}\right)\int_{\mathbb{Z}}\epsilon(u)S_t\,I_t \bar{\pi}(dt,du) +\frac{(dS)^2}{2\,S^2}\nonumber\\
   &+\left(1-\frac{1}{I}\right)( \xi\,S_t\,I_t-(\eta+\gamma)\,I_t)dt+\left(1-\frac{1}{I}\right)\int_{\mathbb{Z}} \epsilon(u)S_{t}\,I_{t}\,\bar{\pi}(dt,du)+\frac{(dI)^2}{2\,I^2}\nonumber\\
   &=\left(1-\frac{\beta}{S}\right)[(\theta-\xi\,S_t\,I_t-\eta\,S_t+\rho \,R_t)dt]-\left(1-\frac{\beta}{S}\right)\int_{\mathbb{Z}}\epsilon(u)S_t\,I_t \bar{\pi}(dt,du) -\beta\,\int_{\mathbb{Z}}\left[\ln(1- \epsilon(u)I_t)+\epsilon(u)\,I_t \right]\nu(du) \nonumber\\
   &+\left(1-\frac{1}{I}\right)( \xi\,S_t\,I_t-(\eta+\gamma)\,I_t)dt+\left(1-\frac{1}{I}\right)\int_{\mathbb{Z}} \epsilon(u)S_{t}\,I_{t}\,\bar{\pi}(dt,du)-\int_{\mathbb{Z}}\left[\ln(1+\epsilon(u)S_t)-\epsilon(u)S_t \right]\nu(du)  \nonumber\\
   &:=LV\,dt+ \phi,
\end{align} 
where 
$$\phi=-\beta\,\int_{\mathbb{Z}}\left[\ln(1- \epsilon(u)I_t)+\epsilon(u)S_t\,I_t \right]\bar{\pi}(dt,du)+\int_{\mathbb{Z}}\left[\epsilon(u)S_t\,I_t-\ln(1+\epsilon(u)S_t) \right]\bar{\pi}(dt,du)$$ and 
$L$ is the operator that acts on a function $V.$ Here, $LV: \mathbb{R}_+^2\rightarrow\mathbb{R}_+$ 
 is defined as follows:
\begin{align}\label{LV}
  LV&=\left(1-\frac{\beta}{S}\right)[(\theta-\xi\,S_t\,I_t-\eta\,S_t+\rho \,R_t)]-\beta\,\int_{\mathbb{Z}}\left[\ln(1- \epsilon(u)I_t)+\epsilon(u)\,I_t \right]\nu(du)\nonumber\\
   &+\left(1-\frac{1}{I}\right)( \xi\,S_t\,I_t-(\eta+\gamma)\,I_t)-\int_{\mathbb{Z}}\left[\ln(1+\epsilon(u)S_t)-\epsilon(u)S_t \right]\nu(du)\nonumber\\
  &=\theta-\xi\,S_t\,I_t-\eta\,S_t+\rho \,R_t+ \xi\,S_t\,I_t-(\eta+\gamma)\,I_t+\frac{\beta\,\theta}{S_t}+\beta\,\xi\,I_t+\beta\,\eta+\beta\,\rho\frac{R_t}{S_t}-\xi\,S_t\nonumber\\
  &+(\eta+\gamma)+\frac{1}{2}\,\int_{\mathbb{Z}} \epsilon^2(u)\,S^2_{t}\,\nu (du)+\frac{1}{2}\,\int_{\mathbb{Z}} \epsilon^2(u)\,I^2_{t}\,\nu (du)\nonumber\\
  &\leq \theta-(\eta+\gamma)\,I_t+\beta\,\xi\,I_t+(\eta+\gamma)+\beta\,\eta+\frac{1}{2}\,\int_{\mathbb{Z}} \epsilon^2(u)\,S^2_{t}\,\nu (du)+\frac{1}{2}\,\int_{\mathbb{Z}} \epsilon^2(u)\,I^2_{t}\,\nu(du)
\end{align}
By assuming $\beta=\frac{\eta+\gamma}{\xi}$, Eq. (\ref{LV}) becomes
\begin{align}\label{LV-2}
    LV&\leq \theta+(\eta+\gamma)+\beta\,\eta-\beta\,\int_{\mathbb{Z}}\left[\ln(1- \epsilon(u)I_t)+\epsilon(u)\,I_t \right]\nu(du)-\int_{\mathbb{Z}}\left[\ln(1+\epsilon(u)S_t)-\epsilon(u)S_t \right]\nu(du)\nonumber\\
    &:= Q >0,
\end{align}
we obtain
\begin{align*}
    dV(S_t,I_t)\leq Q\,dt+\phi
\end{align*}
For more details; see \cite{zhou2016threshold}, Theorem 2.1, \cite{2020A}, Theorem 2.1, and  \cite{zhang2020dynamics}. The proof is therefore completed. $\Box$
\end{proof}
\subsection{Extinction and persistence of the COVID-19 disease with L\'evy jumps }\label{s4}
In this section, our objective is to derive sufficient conditions that elucidate the extinction and persistence of the disease, providing novel insights into the control of COVID-19. We first compute the virus-free equilibrium and endemic equilibrium after establishing the basic reproductive number. Subsequently, we extend our analysis to the stochastic counterpart, determining the basic reproductive number in the stochastic framework. This stochastic reproductive number becomes instrumental in discerning the dynamics of coronavirus extinction and persistence, contributing valuable information for disease control strategies. 

In this section, we aim to identify sufficient conditions for both the extinction and persistence of the disease, providing novel insights that enhance our understanding of COVID-19 control strategies. By exploring the factors that contribute to the disease's extinction or sustained presence, we hope to contribute valuable information to the ongoing efforts in devising effective control measures for COVID-19.
\subsubsection{Reproduction number for stochastic COVID-19 model}
The existence of the virus-free equilibrium and endemic equilibrium hinges on computing the basic reproductive number. In this section, our focus shifts to establishing the basic reproductive number for the stochastic counterpart. This stochastic reproductive number plays a pivotal role in delineating the dynamics of extinction and persistence of the coronavirus (disease). Specifically, our attention is directed towards analyzing the infected class (\ref{stoch1}), i.e.
$$dI_t=( \xi\,S_t\,I_t-(\eta+\gamma)\,I_t)dt+\int_{\mathbb{Z}} \epsilon(u)S_{t^{-}}\,I_{t^{-}}\,\bar{\pi}(dt,du)$$
By selecting a twice-differentiable function $Y_t$ and applying Ito's formula to $Y_t$, we can derive the stochastic basic reproduction number. Let's set $Y_t=\ln{I_t}$, using the natural logarithm of the infectious class. This choice allows us to express the dynamics of the infectious class in a form that facilitates the application of stochastic calculus, leading to the characterization of the stochastic basic reproduction number.

\begin{align}\label{I1}
 dY_t&=\,d\ln{I_t}\nonumber\\
 &=\frac{1}{I_t}dI_t-\frac{1}{2\,I_t^2}(dI_t)^2 \nonumber\\
 &=\frac{1}{I_t}\left[( \xi\,S_t\,I_t-(\eta+\gamma)\,I_t)dt+\int_{\mathbb{Z}} \epsilon(u)S_{t}\,I_{t}\,\bar{\pi}(dt,du)\right]\nonumber\\
 &-\frac{1}{2\,I_t^2}\left[( \xi\,S_t\,I_t-(\eta+\gamma)\,I_t)dt+\int_{\mathbb{Z}} \epsilon(u)S_{t}\,I_{t}\,\bar{\pi}(dt,du)\right]^2\nonumber\\
 &=(\xi\,S_t-(\eta+\gamma))dt+\int_{\mathbb{Z}} \ln(1+\epsilon(u))S_{t}\,\bar{\pi}(dt,du)-\int_{\mathbb{Z}}\left[\ln(1+\epsilon(u)S_t)-\epsilon(u)S_t \right]\nu(du)dt\nonumber\\
 &=\left[\xi\,S_t-(\eta+\gamma)-\int_{\mathbb{Z}}\left[\ln(1+\epsilon(u)S_t)-\epsilon(u)S_t \right]\nu(du)\right]dt+\int_{\mathbb{Z}} \ln(1+\epsilon(u))S_{t}\,\bar{\pi}(dt,du)
\end{align}
Consider the fact that $I=R=0$. $\theta-\eta\,S = 0$, yields $S=\,\frac{\theta}{\eta}$. 


Set

$A=\xi\,S_t- \int_{\mathbb{Z}}\left[\ln(1+\epsilon(u)S_t)-\epsilon(u)S_t \right]\nu(du)=\xi\,\frac{\theta}{\eta}-\int_{\mathbb{Z}}\left[\ln(1+\epsilon(u)\frac{\theta}{\eta})-\epsilon(u)\frac{\theta}{\eta} \right]\nu(du)=\xi\,\frac{\theta}{\eta}-\varphi$ and $B=\eta+\gamma$ 
where $\varphi=\int_{\mathbb{Z}}\left[\ln(1+\epsilon(u)\frac{\theta}{\eta})-\epsilon(u)\frac{\theta}{\eta} \right]\nu(du)\leq \int_{\mathbb{Z}}\left[\ln(1+\epsilon(u)\frac{\theta}{\eta}) \right]\nu(du)\leq K$ from Assumption \ref{as1}.


$$AB^{-1}=\left(\xi\,\frac{\theta}{\eta}-\varphi\right)\,\frac{1}{\eta+\gamma}=\frac{\xi\,\theta}{\eta(\eta+\gamma)}-\frac{\varphi}{\eta+\gamma}$$
Therefore, the reproduction number for the stochastic COVID-19 model (\ref{CV-stoch}) is
\begin{align}\label{P}
     \psi=\psi_0-\frac{\varphi}{\eta+\gamma}
\end{align}
 where $\psi_0=\frac{\xi\,\theta}{\eta(\eta+\gamma)}$ is the reproduction number for deterministic COVID-19 model (\ref{CV-deter}). 
 
This stochastic reproduction number provides a crucial metric for understanding the stochastic dynamics of COVID-19, offering insights into the factors influencing the disease's persistence or extinction.
\begin{remark}
From the equation (\ref{P}), it is evident that $\psi < \psi_0$, highlighting that the stochastic approach is inherently more realistic than its deterministic counterpart. This observation underscores the significance of considering stochastic elements in modeling the dynamics of the system, as it captures the inherent uncertainties and fluctuations that influence the course of the COVID-19 epidemic.
\end{remark}
\begin{figure}[htb!]
    \subfloat[$\Psi \,\,\,VS.\,\,\, \epsilon(u)$]{\includegraphics[width=0.5\textwidth]{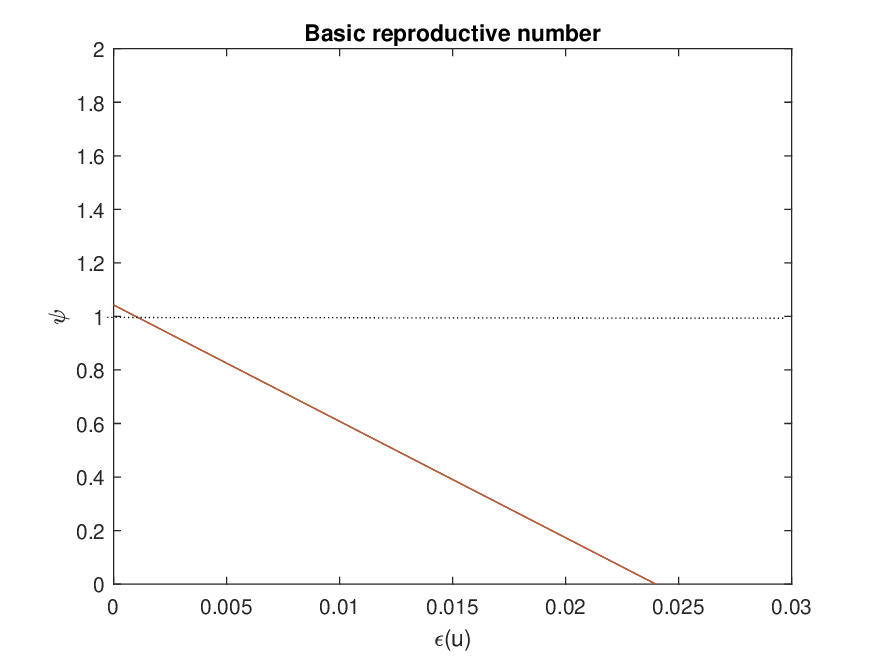}}
    \subfloat[$\Psi\,\,\, VS.\,\,\,\, \theta$]{\includegraphics[width=0.5\textwidth]{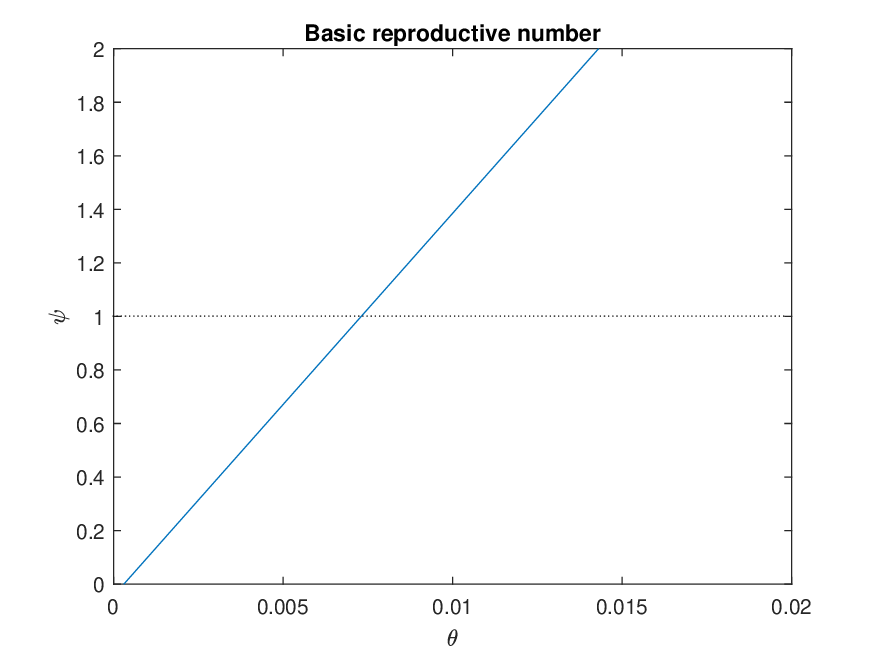}}
    \hfill
    \subfloat[$\Psi\,\,\ VS.\,\,\ \xi$]{\includegraphics[width=0.5\textwidth]{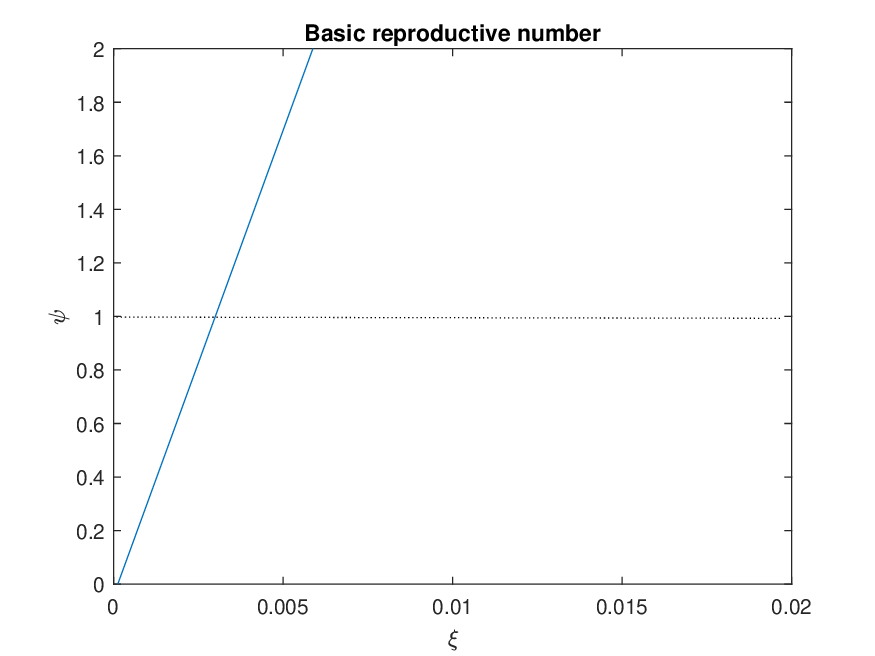}}
    \hfill
    \subfloat[$\Psi \,\,\, VS. \,\,\,\Psi_0$]{\includegraphics[width=0.5\textwidth]{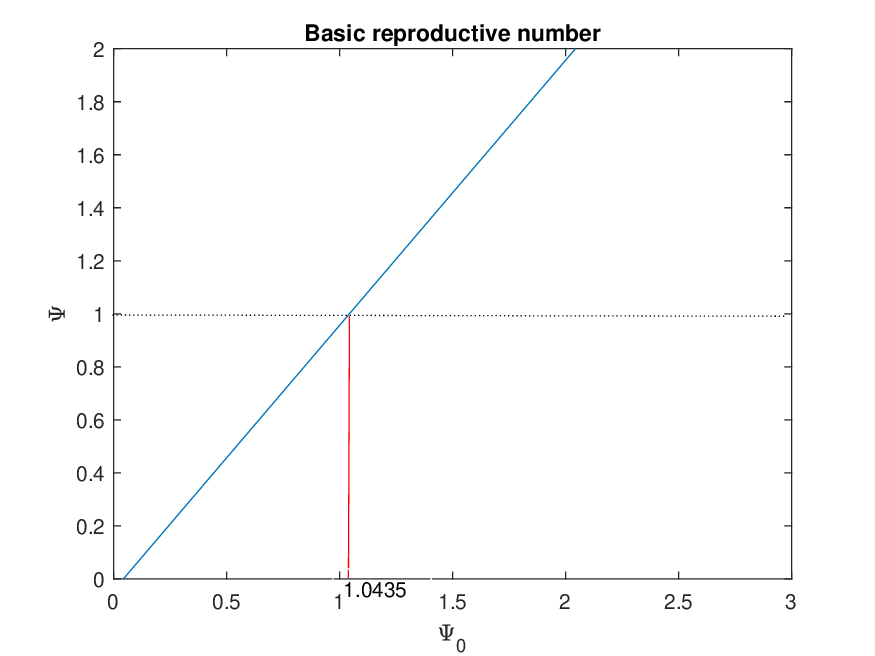}}
\caption{\label{REp-2}This Figure shows the effect of the parameters on the basic reproductive number of the stochastic COVID-19 model(\ref{CV-stoch}) (a) $\theta=0.0073, \,\,\,\xi=0.003,\,\,\eta=0.001,\,\,\rho=0.01,\,\, \gamma=0.02,$ with vary $\epsilon(u)$ (b) 
 $\xi=0.003,\,\,\eta=0.001,\,\,\rho=0.01,\,\, \gamma=0.02,\,\, \epsilon(u)=0.001$ and vary value of $\theta$. (c)  $\theta=0.0073, \,\,\,\xi=0.003,\,\,\eta=0.001,\,\,\rho=0.01,\,\, \gamma=0.02,$ and non-Gaussian noise intensity 
  $\epsilon(u)=0.001$.  (d) 
 $\theta=0.0073, \,\,\,\xi=0.003,\,\,\eta=0.001,\,\,\rho=0.01,\,\, \gamma=0.02,$  \;$\epsilon(u)=0.001$ with vary $\Psi_0$.}
\end{figure}
\subsubsection{Extinction of the COVID-19 disease}
 \begin{theorem}\label{Exit-thr}
Assumption \ref{as1} and $\psi<1$ hold. Then, for any initial condition $(S_0,I_0)\in \mathbb{R}_{+}^2$, the solution $(S_t,I_t) \in \mathbb{R}_{+}^2$ of the stochastic COVID-19 model (\ref{stoch}) satisfies the specified conditions.
\begin{align*}
 lim_{t\rightarrow \infty} sup \frac{\ln I_t}{t}\leq(\eta+\gamma)[\psi-1]<0,\quad a.s.   
\end{align*} 
\end{theorem}
\begin{proof}
Integrating both sides of Eq. (\ref{I1}) from 0 to $t$ yields:
\begin{align}\label{I11}
 \ln I_t=\ln I_0 +\int_{0}^{t}\left[\xi\,S_t-(\eta+\gamma)-\int_{\mathbb{Z}}\left[\ln(1+\epsilon(u)S_t)-\epsilon(u)S_t \right]\nu(du)\right]dt+\int_{0}^{t}\int_{\mathbb{Z}} \ln(1+\epsilon(u))S_{t}\,\bar{\pi}(dt,du)   
\end{align}
Evaluating Eq. (\ref{I11}) at the disease-free equilibrium, we obtain:
\begin{align}\label{I2}
 \ln I_t=&\ln I_0 +\int_{0}^{t}\left[\xi\,\frac{\theta}{\eta}-(\eta+\gamma)-\int_{\mathbb{Z}}\left[\ln\Big(1+\epsilon(u)\frac{\theta}{\eta}\Big)-\epsilon(u)\frac{\theta}{\eta} \right]\nu(du)\right]dt+\int_{0}^{t}\int_{\mathbb{Z}} \ln(1+\epsilon(u))\frac{\theta}{\eta}\,\bar{\pi}(dt,du) \nonumber\\
 &\leq \ln I_0+\int_{0}^{t}\xi\,\frac{\theta}{\eta}-(\eta+\gamma)-\varphi+\int_{0}^{t}\int_{\mathbb{Z}} \ln\Big(1+\epsilon(u)\Big)\frac{\theta}{\eta}\,\bar{\pi}(dt,du)
\end{align}
Multiplying both side of Eq. (\ref{I2}) by $\frac{1}{t}$
\begin{align}\label{I3}
 \frac{1}{t}\ln I_t\leq \frac{1}{t}\ln I_0+ \frac{1}{t}\int_{0}^{t}(\xi\,\frac{\theta}{\eta}-(\eta+\gamma)-\varphi)ds+\frac{1}{t}\int_{0}^{t}\int_{\mathbb{Z}} \ln\Big(1+\epsilon(u)\Big)\frac{\theta}{\eta}\,\bar{\pi}(dt,du)ds
\end{align}
Using the theorem of large numbers and applying $lim _{t\rightarrow \infty}sup$ to Eq. (\ref{I3}) and letting  
$lim _{t\rightarrow \infty} \frac{1}{t}\int_{0}^{t}\int_{\mathbb{Z}} \ln(1+\epsilon(u))\frac{\theta}{\eta}\,\bar{\pi}(dt,du)ds=0$, we get
\begin{align}\label{I4}
lim _{t\rightarrow \infty}sup \frac{1}{t}\ln I_t\leq& \left(\xi\,\frac{\theta}{\eta}-(\eta+\gamma)-\varphi\right)\nonumber\\
&=(\eta+\gamma)\left[\xi\,\frac{\theta}{\eta\,(\eta+\gamma)}-\frac{\varphi}{\eta+\gamma}-1\right]\nonumber\\
&=(\eta+\gamma)[\psi-1]<0 \quad \text{iff}\qquad  \psi<1.
\end{align}
Thus, if $\psi<1$, then the virus-free equilibrium point is locally asymptotically stable. In other words 
$$lim _{t\rightarrow \infty}\frac{I_t}{t}=0.$$
The proof is completed. $\Box$
\end{proof}
\subsection{Persistence of the COVID-19 disease}
In this section, we discuss about the persistence of stochastic model.
\begin{theorem}\label{persis-thr}
Assumption \ref{as1} and $\psi>1$ hold. Then for any initial condition $(S_0,I_0,R_0)\in \mathbb{R}_{+}^3$, the solution $(S_t,I_t,R_t) \in \mathbb{R}_{+}^3$  of the stochastic COVID-19 model (\ref{stoch}) satisfies
\begin{align*}
 lim_{t\rightarrow \infty}<S>_t={S^\star } >0,\quad lim_{t\rightarrow \infty}<I>_t={I^\star } >0, \quad \text{and}\quad lim_{t\rightarrow \infty}<R>_t={R^\star }>  \text{a.s}.   
\end{align*}
where\\
$S^\star=\frac{ \theta}{\eta}-\left( \frac{\eta+\gamma+\rho}{\eta+\rho} \right) (\eta+\gamma)[\psi-1],\qquad I^\star=(\eta+\gamma)[\psi-1] \qquad R^\star=\frac{\gamma}{\eta+\rho}(\eta+\gamma)[\psi-1]$
\end{theorem}
\begin{proof}
Integrating both sides of stochastic COVID-19 model (\ref{CV-stoch}), summing all equations, and multiplying the result by $\frac{1}{t}$, we obtain
\begin{align}\label{S-1}
\frac{S_t-S_0}{t}+\frac{I_t-I_0}{t}+\frac{R_t-R_0}{t}\; =& \; \theta-\xi\,<S>_t\,<I>_t-\eta<S>_t+\rho\,<R>_t-\frac{1}{t}\int_0^t\int_{\mathbb{Z}}\epsilon(u)\,S_r\,I_r\,\bar{\pi}(dr,du)dr\nonumber\\
&+\xi\,<S>_t\,<I>_t-(\eta+\gamma)<I>_t+\frac{1}{t}\int_0^t\int_{\mathbb{Z}}\epsilon(u)\,S_r\,I_r\,\bar{\pi}(dr,du)dr\nonumber\\
&+\gamma\,<I>_t-(\eta+\rho)<R>_t\nonumber\\
&= \theta-\eta<S>_t-\left((\eta+\gamma)-\frac{\rho\,\gamma}{\eta+\rho}\right)<I>_t\nonumber\\
&=\theta-\eta<S>_t-\left( \frac{(\eta+\gamma)(\eta+\rho)-\gamma\,\rho}{\eta+\rho}\right)<I>_t
\end{align}
Now, let us solve $<S>_t$ \;from Eq. (\ref{S-1})
\begin{align}\label{S-2}
  <S>_t \;=\; \frac{1}{\eta}\left[ \frac{S_t-S_0}{t}+\frac{I_t-I_0}{t}+\frac{\rho}{\rho+\eta}\frac{R_t-R_0}{t} \right]+\frac{ \theta}{\eta}-\frac{1}{\eta}\,\left( \frac{(\eta+\gamma)(\eta+\rho)-\gamma\,\rho}{\eta+\rho} \right)<I>_t
\end{align}
Finding the limit as $t$ tends to infinity for equation (\ref{S-2})
\begin{align}\label{S-3}
    lim_{t\rightarrow\infty}<S>_t=\frac{ \theta}{\eta}-\frac{1}{\eta}\,\left( \frac{(\eta+\gamma)(\eta+\rho)-\gamma\,\rho}{\eta+\rho} \right)<I>_t
\end{align}
Substituting Eq. (\ref{I4}) in to Eq. (\ref{S-3}), yields

\begin{align}
    lim_{t\rightarrow\infty}<S>_t&=\frac{ \theta}{\eta}-\frac{1}{\eta}\,\left( \frac{(\eta+\gamma)(\eta+\rho)-\gamma\,\rho}{\eta+\rho} \right) (\eta+\gamma)[\psi-1]\nonumber\\
    &=\frac{ \theta}{\eta}-\left( \frac{\eta+\gamma+\rho}{\eta+\rho} \right) (\eta+\gamma)[\psi-1]
\end{align}
Next apply $lim_{t\rightarrow\infty}$ in to the following Equation
\begin{align}\label{R-1}
    \frac{R_t-R_0}{t}=\gamma\,<I>_t-(\eta+\rho)<R>_t
\end{align}
Replace $<I>_t$ as in Eq. (\ref{I4}), gives 
\begin{align}\label{R-2}
   lim_{t\rightarrow\infty}<R>_t=\frac{\gamma}{\eta+\rho}(\eta+\gamma)[\psi-1]
\end{align}
The proof of the theorem is now complete. $\Box$
\end{proof}
\begin{figure}[htb!]
\subfloat[The Susceptible]{\includegraphics[width=0.5\textwidth]{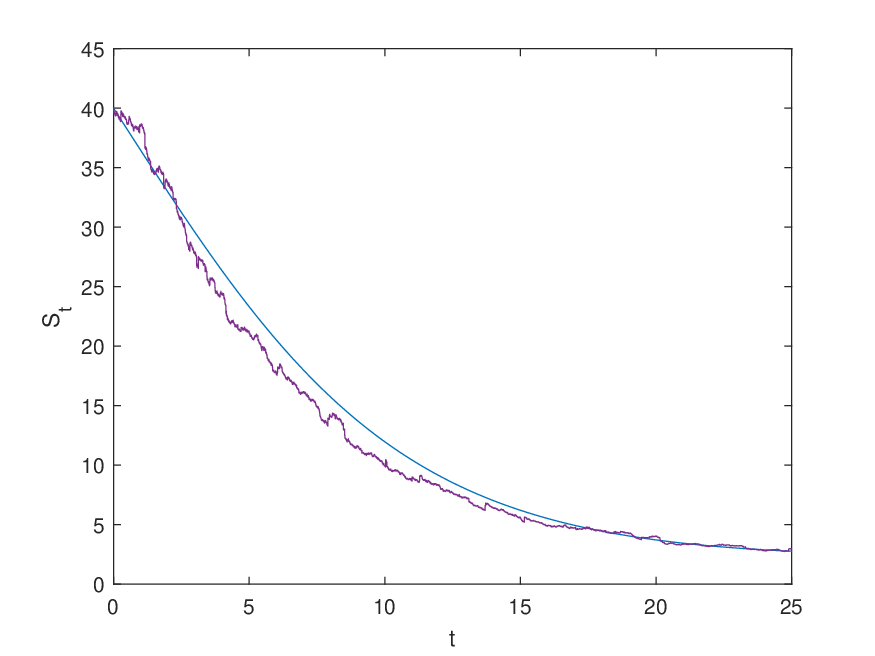}}
\subfloat[The Infected population by COVID-19]{\includegraphics[width=0.5\textwidth]{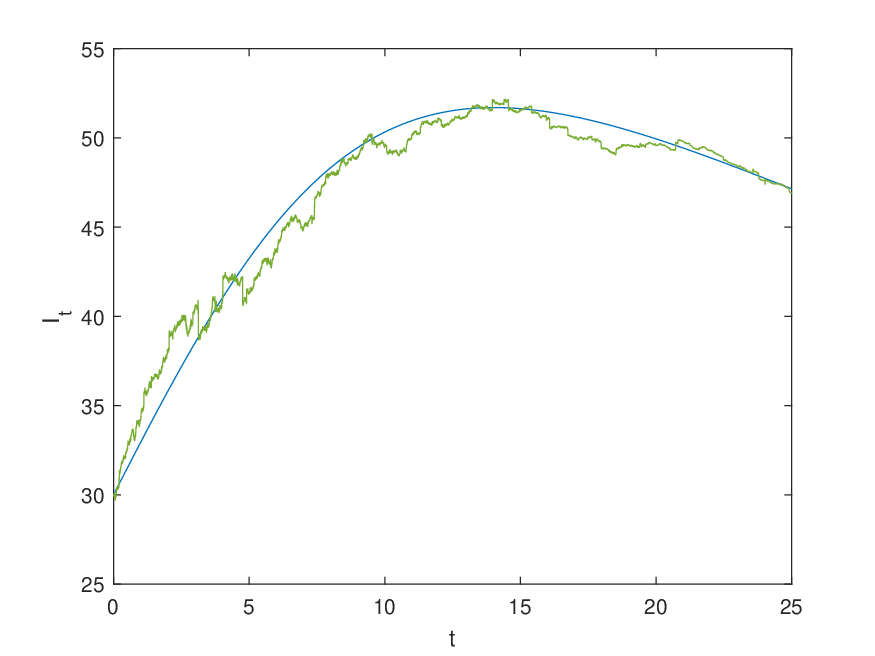}}
\caption{\label{SIR-2} The numerical simulation of model (\ref{stoch}) is presented in the above graphs. (a) The graph depicting the susceptible population. (b) The graph illustrating the number of infected individuals. Parameters $\theta=0.0073, \,\,\,\xi=0.003,\,\,\eta=0.001,\,\,\rho=0.01,\,\, \gamma=0.02,\,\ \psi=0.9994 <1,$  and non-Gaussian noise intensity $\epsilon(u)=0.001.$  }
\end{figure}
\begin{fact}
The biological implication, as discussed in \cite{2020Mathematical} and supported by Theorem \ref{persis-thr}, is that a minimal influx of infectious individuals is unlikely to lead to a disease outbreak unless $\psi > 1.$
\end{fact}

\section{Numerical experiments}\label{Numb}
Numerical solutions of systems are invaluable in the study of epidemic models. This section presents the numerical results of our model, shedding light on how the parameters of the deterministic model (\ref{CV-deter}) and the intensity of non-Gaussian noise in the stochastic model (\ref{CV-stoch}) impact the dynamics. We conduct numerical experiments to illustrate the extinction and persistence of the novel coronavirus, COVID-19, in both the deterministic model and its corresponding stochastic system for comparison.

We set the parameter values as follows: $\theta=0.0073,\,\,\,\xi=0.003,\,\,\eta=0.001,\,\,\rho=0.01,\,\, \gamma=0.02,\,$ and non-Gaussian noise intensity $\epsilon(u)=0.001$ in the deterministic model (\ref{determ}) and stochastic model (\ref{stoch1}). Then, the system (\ref{CV-deter}), becomes
\begin{align*}
 &\frac{dS_t}{dt}= 0.0073-0.003\,S_t\,I_t-0.001\,S_t+0.01 \,R_t\nonumber\\
 &\frac{dI_t}{dt}= 0.003\,S_t\,I_t-(0.001+0.02)\,I_t\nonumber\\
 &\frac{dR_t}{dt}=0.02\,I_t-(0.001+0.01)\,R_t
\end{align*}
Using simple calculations, we find that $\psi_0=1.0429 >1$ and $\psi=0.9994 <1.$ According to Theorem \ref{Exit-thr}, the deterministic model (\ref{CV-deter}) suggests that the disease will prevail in the long run. However, based on Theorem \ref{persis-thr}, the stochastic model (\ref{stoch1}) implies that the disease goes to extinction.

From Figure \ref{REp-2}, it is observed that if the noise intensity is large, the extinction of COVID-19 occurs. For $\theta \in (0,0.0073)$,\,\, $\xi \in (0,0.003)$, and $\psi_0 \in (0,1.0435),$ the coronavirus-2 (COVID-19) goes to extinction.

The effects of non-Gaussian noise (jumps) on system (\ref{stoch1}) are plotted in Fig. \ref{SIR-2}. $\psi_0 >1$ indicates that the disease prevails in the deterministic model case, but in the stochastic model case, $\psi < 1$, implying the disease dies out. This result indicates that the stochastic approach is more realistic than the deterministic approach. In other words, from this result, we can see that $\psi_0 > \psi$ because of the effect of the jump noise. 

Next, let us keep the parameters the same as shown in Figure \ref{SIR-2}(a), but $\xi=0.0033$. Figure \ref{SIR-2}(b) is plotted based on Theorem \ref{persis-thr}, i.e., for  $\psi >1$, then 
\begin{align*}
 &lim_{t\rightarrow \infty}<S>_t={S^\star }=7.2977 >0,\nonumber\\
 &lim_{t\rightarrow \infty}<I>_t={I^\star }=0.0022 >0, \nonumber\\
 &\text{and}\quad lim_{t\rightarrow \infty}<R>_t={R^\star }=0.0015 >0 \quad \text{a.s}.   
\end{align*}
From this figure, it is evident that $I_t$ persists.

\section{Conclusion}\label{Conc}

This paper delves into the intricate dynamics of a COVID-19 epidemic model augmented with non-Gaussian noise. The exploration extends beyond the deterministic facet of the model, leading to a rigorous proof of the existence and uniqueness of a non-negative global solution for the stochastic system (\ref{CV-stoch}). These novel findings not only contribute to the theoretical foundation but also augment and refine insights garnered from preceding studies, as perceptibly illustrated in the graphical representations shown throughout this manuscript.

The deterministic analysis provides a foundational understanding, while the stochastic counterpart offers a nuanced perspective, acknowledging the inherent uncertainties and fluctuations that characterize real-world epidemiological scenarios. The demonstration of the existence and uniqueness of solutions in the stochastic framework underscores the robustness of the model in capturing the complexities of COVID-19 dynamics.

A pivotal aspect of this research lies in the identification of sufficient conditions for the extinction of COVID-19 (\ref{CV-stoch}) when $\psi <1$. This implication is crucial in the realm of disease control, suggesting that if the basic reproduction number remains below unity, the disease is prone to die out. Such insights can guide public health strategies, aiding in the design of effective measures to curb the spread of COVID-19.

The graphical representations serve as visual corroboration of the theoretical findings, providing a tangible means to comprehend the interplay of parameters and the consequential impact on disease dynamics. This bridge between theory and visualization enhances the accessibility and applicability of the research findings.

In conclusion, this study significantly contributes to the ongoing discourse on the modeling and understanding of COVID-19 dynamics. The theoretical advancements, coupled with practical implications, offer a comprehensive framework for researchers, policymakers, and healthcare professionals striving to navigate and mitigate the impact of the COVID-19 pandemic.

\renewcommand{\bibname}{references} 
\bibliographystyle{ieeetr}
\bibliography{references.bib}
\end{document}